\documentclass[11pt]{article}
\usepackage[latin1]{inputenc}

\usepackage{amsfonts}
\usepackage{amssymb}
\usepackage{amsthm}
\usepackage{amsmath}
\usepackage{graphicx}
\usepackage{color}
\usepackage{mathrsfs}
\usepackage{amscd}
\usepackage{multirow}
\graphicspath{ {images/} }

\newcommand{\CC}{{\mathbb C}}
\newcommand{\C}{{\mathbb C}}
\newcommand{\cD}{{\mathcal D}}

\newcommand{\cO}{{\mathcal O}}

\newcommand{\cM}{{\mathcal M}}

\newcommand{\R}{{\mathbb R}}
\newcommand{\Z}{{\mathbb Z}}
\newcommand{\NN}{{\mathbb N}}

\newcommand\w{{\omega}}
\newcommand\N{{\mathbb{N}}}

\newcommand\D{{\mathcal{D}}}
\newcommand\Q{{\mathbb{Q}}}

\newcommand\M{{\mathcal{M}}}

\newtheorem{theorem}{Theorem}[section]
\newtheorem{prop}[theorem]{Proposition}
\newtheorem{defi}[theorem]{Definition}
\newtheorem{example}[theorem]{Example}
\newtheorem{lemma}[theorem]{Lemma}
\newtheorem{cor}[theorem]{Corollary}
\newtheorem{remark}[theorem]{Remark}
\newtheorem{notation}[theorem]{Notation}
\newtheorem{conjecture}[theorem]{Conjecture}

\makeatletter
\newcommand{\subjclass}[2][1991]{%
  \let\@oldtitle\@title%
  \gdef\@title{\@oldtitle\footnotetext{#1 \emph{Mathematics subject classification.} #2}}%
}
\newcommand{\keywords}[1]{%
  \let\@@oldtitle\@title%
  \gdef\@title{\@@oldtitle\footnotetext{\emph{Key words and phrases.} #1.}}%
}
\makeatother

\begin{document}
\title{On the local monodromy of $A$--hypergeometric functions and some monodromy invariant subspaces.}

\author{Mar\'ia-Cruz Fern\'andez-Fern\'andez\thanks{Supported by Ministerio de Econom\'ia y Competitividad
MTM2013-40455-P, MTM2016-75024-P, Feder, P12-FQM-2696 and FQM-333.}}

\maketitle
\date{}


\begin{abstract}
We obtain an explicit formula for the characteristic polynomial of the local monodromy of $A$--hypergeometric functions 
with respect to small loops around a coordinate hyperplane $x_i =0$. This formula is similar to the one obtained by 
Ando, Esterov and Takeuchi for the local monodromy at infinity. Our proof is combinatorial and can be adapted to provide an alternative proof for the latter formula as well. 
On the other hand, we also prove that the solution space at a nonsingular point of certain irregular and irreducible $A$--hypergeometric $D$--modules has a nontrivial global monodromy invariant subspace.
\end{abstract}

\section{Introduction}
Gel'fand, Graev, Kapranov and Zelevinsky started the study of $A$--hypergeometric systems in \cite{GGZ} and \cite{GKZ}. 
These systems of linear partial differential equations generalize all of the classical hypergeometric equations and have many applications in other areas of Mathematics. They are 
determined by a matrix $A=(a_1 \cdots a_n )\in \Z^{d\times n}$ such that $\Z A :=\sum_{j=1}^n \Z a_i\simeq \Z^d$ 
and a parameter vector $\beta \in \C^d$. More precisely, let $H_A (\beta )$ be the left ideal of the Weyl algebra $D=\C [x_1 ,\ldots ,x_n ]\langle \partial_1 ,\ldots ,\partial_n \rangle$ generated by the following
set of differential operators:

\begin{equation}
\Box_u := \partial^{u_+}-\partial^{u_{-}} \; \;\; \; \mbox{ for }
u\in \Z^n , \; Au=0 \label{Toric-operators}
\end{equation} where $u=u_+-u_-$ and  $u_+, u_- \in \NN^n$ have disjoint supports, and

\begin{equation}
E_i - \beta_i := \sum_{j=1} ^n a_{ij}x_j\partial_j -\beta_i \; \; \;
\; \mbox{ for } i=1,\ldots, d. \label{Euler-operators}
\end{equation}

The $A$--hypergeometric $D$--module with parameter $\beta$ is $M_A(\beta)=D /D H_A (\beta )$. One can also consider its analytification version $\cM_A (\beta )=\cD/\cD H_A (\beta)$, where $\cD$ denotes the sheaf of linear partial differential operators 
with coefficients in the sheaf $\cO$ of holomorphic functions on $\C^n$. Such a $\cD$--module was proved to be holonomic in \cite{GGZ} and \cite{Ado}. Moreover, if $\beta$ is nonresonant for $A$ (i.e., the boundary of $\sum_{i=1}^n \R_{\geq 0} a_j$ does not contain any point of $\beta +\Z A$) they also proved that the holonomic rank of $\cM_A (\beta )$ (i.e. the dimension of its space of holomorphic solutions) 
equals the normalized volume of $A$ (see (\ref{nvol})). The exact set of parameters for which this is true was given in \cite{MMW}. It is also well known that the $A$--hypergeometric system is regular holonomic if and only if the $\Q$--rowspan of $A$ contains the vector $(1,1,\ldots, 1)$ (see \cite[Section 6]{Hotta},\cite[Theorem 2.4.11]{SST},\cite[Corollary 3.16]{SW}).

One fundamental open problem in this setting is to understand the monodromy of the solutions of a general $A$--hypergeometric $D$--module. By \cite[Thm 2.11]{GKZ2} if $\beta\in \CC^d$ is nonresonant and $M_A (\beta)$ is regular holonomic then the monodromy representation of its solutions is irreducible.
Adolphson conjectured this fact to be true also in the irregular case (see the comment in \cite{Ado} after Corollary 5.20), but we will see in this note that this is not the case (see Corollary \ref{one-dimensional-invariant}).
On the other hand, it was proved in \cite{Saito} that $M_A (\beta)$ is an irreducible $D$--module if and only if $\beta$ is nonresonant. 
Independently, in \cite{SW2} the set of parameters $\beta$ for which $\C (x)\otimes_{\C [x]} M_A (\beta)$ is an irreducible $\C (x)\otimes_{\C [x]} D$--module is characterized, generalizing \cite{Beu2}. It was previously proved in \cite{Wal07}
that the irreducibility of $\C (x)\otimes_{\C [x]} M_A (\beta)$ depends only on the equivalence class of $\beta \in \C^d$ modulo $\Z A :=\sum_{j=1}^n \Z a_j$. Let us point out that $M$ being an irreducible $D$--module implies 
$\C (x)\otimes_{\C [x]} M$ being an irreducible $\C (x)\otimes_{\C [x]} D$--module. Moreover, for a regular holonomic $\cD$--module $\cM$, it is equivalent to say that $\cM$ is an irreducible $\cD$--module and that 
its solution sheaf complex $\R Hom_{\cD}(\M, \cO )$ is an irreducible perverse sheaf by the Riemann--Hilbert correspondence of Kashiwara \cite{Kas84} and Mebkhout \cite{Meb84}.

In some special cases of regular $A$--hypergeometric systems, Beukers provided a method to compute the monodromy group of the solutions of $M_A (\beta)$ \cite{Beu}. Monodromy of regular bivariate hypergeometric systems of Horn type is also investigated in \cite{ST}. 
On the other hand, in \cite{Tak} and \cite{AET} the authors provide a formula for the characteristic polynomial of the local monodromy at infinity of the $A$--hypergeometric functions with nonresonant parameters, that is, with respect to large enough loops around $x_j=0$ for $j=1,\ldots ,n$. 
The proof of this result in \cite{AET} is based on the use of rapid decay homology cycles constructed in \cite{ET}. 
One goal of this paper is to obtain a similar formula for the corresponding local monodromy with respect to small enough loops around $x_j=0$. To this end, 
we first characterize in Section \ref{Sec2} those regular triangulations of $A$ that yield a basis of convergent $\Gamma$--series solutions of the $A$--hypergeometric system in a common open set containing this type of loops. 
We will see that they correspond to refinements of a particular regular polyhedral subdivision of $A$. 
In Section \ref{Sec3} we recall the construction of $\Gamma$--series solutions of $A$--hypergeometric systems, introduced in \cite{GKZ}.
In Section \ref{Sec4} we obtain the formula for the characteristic polynomial of the local monodromy around $x_j=0$ which depends only on this polyhedral subdivision (see Theorem \ref{main-theorem}). 
Our proof can be easily adapted to obtain the monodromy at infinity (see Remark \ref{remark-infinity}), providing a simpler proof of the main result in \cite{AET}. 
In Section \ref{Sec5}, we conjecture the existence of a global monodromy invariant subspace of solutions of $M_A(\beta)$ (see Conjecture \ref{monodromy-invariant-subspaces-conjecture}) and we prove it under certain additional condition (see Proposition \ref{monodromy-invariant-subspaces}). We also show that a proof of the conjecture would characterize   
when the solution space of $M_A (\beta)$ at a nonsingular point has reducible (global) monodromy representation (i.e., when it has a nontrivial monodromy invariant subspace). In particular, there is a family of irregular 
$A$--hypergeometric systems with nonresonant parameters whose solutions spaces at a nonsingular point are direct sums of one--dimensional monodromy invariant subspaces, 
despite the fact that $M_A (\beta)$ is an irreducible $D$--module in this case (see Corollary \ref{one-dimensional-invariant}).

The author is grateful to Saiei-Jaeyeong Matsubara-Heo for pointing out a gap in a previous version of Section \ref{Sec5}.

\section{On certain regular triangulations of $A$}\label{Sec2}

In this section we recall the definition of regular triangulation of a full rank matrix $A=(a_1  \cdots a_n )\in \Z^{d\times n}$, consider certain open sets in $\CC^n$ associated to them and determine those regular triangulations that will be useful for local monodromy computations.

For any set $\tau \subseteq \{1,\ldots ,n\}$ let $\Delta_{\tau}$ be the convex hull of $\{ a_i : \; i\in \tau \}\cup \{\mathbf{0} \}\subseteq \R^d$. In order to simplify notation, we 
shall identify $\tau$ with the set $\{a_i : \; i \in \tau \}$ and with its convex hull and denote by $A_{\tau}$ the corresponding submatrix of $A$. Let us denote $\Z \tau = \Z A_{\tau} =\sum_{i\in
\tau} \Z a_i \subseteq \Z^d$ and $\operatorname{pos}(\tau ):=\sum_{i\in \tau} \R_{\geq 0} a_i
\subseteq \R^d$. We will also denote $\overline{\tau}=\{1, \ldots ,n \}\setminus \tau$.

We will assume for simplicity that $\Z A =\Z^d$ throughout this paper.

A vector $\w =(\w_1 ,\ldots ,\w_n ) \in \R^n$ defines an abstract polyhedral complex
$\mathrm{T}_{\w}$ with vertices in $\{1, \ldots , n\}$ as follows:  $\tau \in \mathrm{T}_{\omega}$ iff there
exists a vector $\mathbf{c}\in \R^d$ such that
\begin{equation}
\langle \mathbf{c} , a_j \rangle =\omega_j  \mbox{ for all } j\in \tau \label{equality-subdivision}
\end{equation}
\begin{equation}
\langle \mathbf{c} , a_j \rangle < \omega_j \mbox{ for all } j \notin \tau \label{inequality-subdivision}.
\end{equation}

We will say that $\w$ is a weight vector and that $T_{\w}$ is a regular subdivision of $A$ if $\operatorname{pos}(A)=\cup_{\tau \in T_{\w}} \operatorname{pos}(\tau)$. This is always the case if either $\w_i >0$ for all $i=1,\ldots, n$ or $A$ is pointed (i.e. the intersection of $\R^n_{>0}$ with the 
$\Q$-rowspan of $A$ is nonempty). If $T$ is any regular subdivision of $A$ then the set $C(T)=\{ \w \in \R^n : \; T=T_{\w}\}$ is a convex polyhedral cone. 
The closures of these cones form the socalled secondary fan of $A$, introduced and studied by Gelfand, Kapranov and Zelevinsky \cite{GKZ3} (see also \cite{Sturm}).

\begin{remark}
If we take $\w =(1, \ldots ,1)$ then $T_{\w}$ is the set of facets of $\Delta_A$ that do not contain the origin. Let us denote $\Gamma_A :=T_{\w}$ in this case.
\end{remark}

\begin{defi}
A weight vector $\omega \in \R^n$ is said to be generic if $\mathrm{T}_{\omega}$ is an abstract simplicial complex. In this case, $T_{\w}$ is called a regular triangulation of $A$.
\end{defi}

\begin{remark}\label{remark-refinement}
A weight vector $\w$ defines a regular triangulation $T_{\w}$ if and only if $C(T_{\w})$ is a full-dimensional cone in the secondary fan of $A$. On the other hand, a vector $\w'$ belongs to the closure of the cone $C(T_{\w})$ if and only if $T_{\w}$ is a refinement of $T_{\w '}$.
\end{remark}

A set $ \sigma \subseteq \{1, \ldots , n\}$ is called a simplex if the columns of $A_{\sigma}$ form a basis of $\R^d$. For any simplex $\sigma$ we set 
\begin{equation}
U_{\sigma} (R):= \{x\in \C^n : \;  \; |x_{j}|< R
|x_{\sigma}^{A_{\sigma}^{-1}a_{j}}|, \; \forall j \notin \sigma \mbox{ such that}\; |A_{\sigma}^{-1}a_{j}|=1 \} \subseteq \CC^n \label{open-set-simplex}\end{equation} 
where $R>0$, $|b|$ denotes the sum of the coordinates of $b\in \R^d$, $x_{\sigma}=(x_i : \; i\in \sigma)$ and we use the multi-index notation for $x_{\sigma}^{A_{\sigma}^{-1}a_{j}}$.

\begin{remark}\label{remark-open-set}
For any $R>0$ the open set $U_T(R):=\cap_{\sigma \in T } U_{\sigma}(R)$ is not empty for any regular triangulation $T$ of $A$ since it contains those
points $x\in (\C^{\ast})^n$ for
which $(-\log |x_1 |,\ldots ,-\log |x_n |)$ lies in a sufficiently
far translation of the nonempty open cone $C(T)$ inside itself (see \cite[Proposition 2 and Section 1.2]{GKZ} and \cite[Remark 6.1]{Fer}).
\end{remark}

\begin{notation}\label{regular-subdivision-T0}
Let us denote $\w_0=\w_0 (\varepsilon):=(1, \ldots , 1)+ \varepsilon  (0,\ldots , 0, 1) $ for $\varepsilon >0$ small enough so that we have the equality of polyhedral subdivisions $T_{w_0(\varepsilon ')}=T_{\w_0 (\varepsilon)}$ for all $\varepsilon ' \in (0,\varepsilon)$. We will
also denote by $T_0$ this regular subdivision.
\end{notation}

\begin{remark}\label{T0description}
If $a_n$ is not a vertex of $\Delta_A$ then $T_{0}=\{\tau \setminus \{n \}: \; \tau \in \Gamma_A \}$. Let us assume that $a_n$ is a vertex of $\Delta_A$. 
If $a_n \notin \tau \in \Gamma_A$ then $\tau \in T_{0}$. If $a_n \in \tau \in \Gamma_A$, let $\tau '$ 
be the convex hull of all the columns of $A$ in $\tau$ but $a_n$. Then $\tau'\in T_{0}$. The rest of facets of $T_0$ are of the form 
$\Gamma \cup \{a_n \} \in T_{0}$ for any facet $\Gamma$ of $\tau '$ that is not contained in a facet of $\tau$ (see Figure 1 in Example \ref{example}). 
\end{remark}

The following lemma is a key ingredient in the proof of Theorem \ref{main-theorem}.

\begin{lemma}\label{key-Lemma}
Let $T$ be a regular triangulation of $A$ that refines $\Gamma_{A}$. The following conditions are equivalent:
\begin{enumerate}
\item[i)] $\w_0$ belongs to the closure of the cone $C(T)$.
 \item[ii)] $T$ refines the polyhedral subdivision $T_0$.
 \item[iii)] $U_T (R)\cap \{x_n =0\}\neq \emptyset$ for all $R>0$.
 \item[iv)] $U_T (R)\cap \{x_n =0\}\neq \emptyset$ for some $R>0$.
\end{enumerate}
\end{lemma}

\begin{proof} The equivalence of i) and ii) is just a particular case of Remark  \ref{remark-refinement} and it is obvious that iii)$\Rightarrow$ iv).
Let us prove first ii)$\Rightarrow$ iii). Since $T$ is a regular triangulation that refines $T_0$, there exists $\w'\in \mathbb{R}^n$ such that $T=T_{\w}$ with $\w=\w_0 (\varepsilon ) + \varepsilon^2 \w'$ for $\varepsilon>0$ small enough.
Fix any $R>0$. We have that $U_T := U_T(R) \subseteq \CC^n$ is a nonempty open set. Thus, we can choose $p=(p_1 ,\ldots , p_n )\in U_T $ such that $p_j\neq 0$ for $1\leq j <n$. Let us prove that $\pi_n (p):=(p_1 , \ldots , p_{n-1} ,0) \in U_T $. 
Since $p\in U_T $ it is clear that $\pi_n (p)$ satisfies all the inequations in (\ref{open-set-simplex}) that do not depend on $x_n$. 
Obviously, it also satisfies all the inequations of the form $|x_n | < R
|x_{\sigma}^{A_{\sigma}^{-1}a_{n}}|$. Thus we only need to check that $\pi_n (p)$ satisfies the inequations $|x_{j}|< R
|x_{\sigma}^{A_{\sigma}^{-1}a_{j}}|$ for all $ j \notin \sigma$ whenever $n \in \sigma \in T$ and $|A_{\sigma}^{-1}a_j|=1$. It is enough to show that the last coordinate of $A_{\sigma}^{-1}a_{j}$ is either zero or negative in this case.
Since $\sigma \in T$ there exists $\tau \in T_{0}$ such that $\sigma \subseteq \tau$. Thus, 
there exists a vector $\mathbf{c}\in \R^d$ satisfying (\ref{equality-subdivision}) and (\ref{inequality-subdivision}) for the vector $\w_0$. Since $\sigma\subseteq \tau$ is a simplex we obtain from  (\ref{equality-subdivision}) that 
$\mathbf{c}=(1,\ldots , 1 , 1+ \varepsilon) A_{\sigma}^{-1}$. Thus, by (\ref{inequality-subdivision}) we have that $(1,\ldots , 1 , 1+ \varepsilon) A_{\sigma}^{-1}a_j \leq 1$ for all $j\notin \sigma$ for all $\varepsilon >0$ small enough. For all $j\notin \sigma$ such that $|A_{\sigma}^{-1}a_j|=1$, 
this implies that the last coordinate of $A_{\sigma}^{-1}a_j$ is either zero or negative. As a consequence, $\pi_n (p) \in U_T$.

Let us prove now iv)$\Rightarrow$i). Since $C(T)$ is a convex cone, it is enough to show that if $\w \in C(T)$ then $\w + \w_0 \in C(T)$. Notice that $\w \in C(T)$ if and only if $T=T_{\w}$ and from (\ref{equality-subdivision}) and (\ref{inequality-subdivision}) 
this holds exactly when $\w_{\sigma}A_{\sigma}^{-1} a_j < \w_j$ for all $j\notin \sigma$ and for all $\sigma \in T$. Since $T$ refines $\Gamma_A$, we have $|A_{\sigma}^{-1}a_j|\leq 1$ for all $j\notin \sigma$, $\forall \sigma \in T$. 
Since $U_T (R)\cap \{x_n =0\}\neq \emptyset$ for some $R>0$, we have that for all $\sigma \in T$, $\forall j \notin \sigma$ such that $|A_{\sigma}^{-1}a_j|=1$ then the last coordinate of $A_{\sigma}^{-1}a_j$ must be nonpositive if $n\in \sigma$.  
Hence, we have $(\w_0)_{\sigma} A_{\sigma}^{-1}a_j \leq 1 \leq \w_{0,j}$ for all $j\notin \sigma$ and thus $(\w + \w_0)_{\sigma} A_{\sigma}^{-1} a_j < \w_j + \w_{0,j}$ for all $j\notin \sigma$ and for all $\sigma \in T$. This implies that $\w +\w_0 \in C(T)$.
\end{proof}

\section{$\Gamma$--series solutions of $\M_A (\beta )$}\label{Sec3}

For any set $\tau \subseteq \{1,\ldots ,n\}$, we recall that the normalized volume of $\tau$ (with respect to the lattice $\Z A =\Z^d$) is given by:
\begin{equation}
\operatorname{vol}_{\Z^d}(\tau)
=d! \operatorname{vol}(\Delta_{\tau})\label{nvol}\end{equation} where
$\operatorname{vol}(\Delta_{\tau})$ denotes the Euclidean volume
of $\Delta_{\tau}$.
If $\sigma\subseteq \{1,\ldots ,n\}$ is a simplex, the normalized volume of $\sigma$ with
respect to $\Z^d$ is equal to $[\Z^d : \Z A_{\sigma}]=|\det (A_{\sigma})|$.

For $v\in \C^n$ with $A v =\beta$ the $\Gamma$--series defined in
\cite{GKZ}:

$$\varphi_v : =\sum_{u\in L_A} \frac{1}{\Gamma (v+u+1)} x^{v+u}$$ is
formally annihilated by the differential operators
(\ref{Toric-operators}) and (\ref{Euler-operators}). Here $\Gamma$
is the Euler Gamma function and $L_A :=\ker (A)\cap \Z^n$. These series were used in \cite{GKZ} in order to construct bases 
of holomorphic solutions of $\M_A (\beta )$ in the case when all the columns of $A$ belong to the same hyperplane, but they can be used in the 
general case too (see for example \cite{Ohara-Takayama}, \cite{Fer}, \cite{DMM}).

Let $v_{\sigma}^{\mathbf{k}}\in \CC^n$ be the vector satisfying $A v_{\sigma}^{\mathbf{k}}=\beta$ and $(v_{\sigma}^{\mathbf{k}})_j = k_j$ for $j\notin \sigma$, where
$\mathbf{k}=(k_i )_{i\notin \sigma }\in \N^{\overline{\sigma}}.$

We consider the series: $$\phi_{\sigma}^{\mathbf{k}}:=\varphi_{v_{\sigma}^{\mathbf{k}}}=x_{\sigma}^{A_{\sigma}^{-1}\beta}\sum_{\mathbf{k}+
\mathbf{m}\in \Lambda_{\mathbf{k}}} \frac{x_{\sigma}^{-
A_{\sigma}^{-1}(\sum_{i\notin \sigma} (k_i +m_i ) a_i)}
x_{\overline{\sigma}}^{\mathbf{k} +\mathbf{m}}}{\Gamma
(A_{\sigma}^{-1}(\beta -\sum_{i\notin \sigma} (k_i +m_i )
a_i)+\mathbf{1})(\mathbf{k}+\mathbf{m})!}
$$ where $$
\Lambda_{\mathbf{k}} :=\{\mathbf{k}+\mathbf{m}=(k_i + m_i )_{i\in
\overline{\sigma}}\in \N^{n-d}: \; \sum_{i\in \overline{\sigma}}
a_{i} m_i \in \Z A_{\sigma} \}.$$ Notice that
$\phi_{\sigma}^{\mathbf{k}}$ is zero if and only if for all
$\mathbf{m}\in \Lambda_{\mathbf{k}} $, $A_{\sigma}^{-1}(\beta -
\sum_{i\notin \sigma } (k_i + m_i ) a_i)$ has at least one negative
integer coordinate.

Let $$\Omega_{\sigma}\subseteq \NN^{\overline{\sigma}}$$ be a set of representatives for the different classes with respect to the following equivalence relation in $\NN^{\overline{\sigma}}$: we say that $\mathbf{k}\sim \mathbf{k'}$ if and only if $A_{\overline{\sigma}}\mathbf{k}-A_{\overline{\sigma}} \mathbf{k'}\in \Z A_{\sigma}$. 
Thus, $\Omega_{\sigma}$ is a set of cardinality $\operatorname{vol}_{\Z^d} (\sigma )$.

Given a regular triangulation $T$ of $A$ we will say that $\beta $ is very generic if 
$A_{\sigma}^{-1}(\beta - \sum_{i\notin \sigma} k_i a_i)$ does not
have any integer coordinate for all $\sigma \in T$ and for 
all $\mathbf{k} \in \Omega_{\sigma}$. Thus, very generic parameter vectors
$\beta$ lie in the complement of a countable union of hyperplanes.

If $\beta \in \CC^d$ is very generic, the series $\phi_{\sigma}^{\mathbf{k}}$ is convergent if and only if $\sigma$ is contained in a facet of $\Gamma_A$ (see for example \cite[Corollary 3.9]{Fer}). In this case, it is convergent in $U_{\sigma}(R)\cap ((\C^{\ast})^{\sigma}\times \C^{\overline{\sigma}})$ for some $R>0$. 
Moreover, we have the following result which is a slightly modified version of \cite[Theorem 2]{Ohara-Takayama} (this version can also be seen as the particular case $\tau =A$ in \cite[Section 6.2]{Fer} by substituting ``Gevrey series along $Y_{\tau}$'' by 
``convergent series''). The fact that our open set $U_T (R)$ is defined by less restrictions than the ones used in \cite{Ohara-Takayama} is important for the results in Section \ref{Sec5}.

\begin{theorem}\label{key2theorem}
If $T$ is a regular triangulation of $A$ that refines 
$\Gamma_A$ and $\beta \in \CC^d$ is very generic, the set $\{\phi_{\sigma}^{\mathbf{k}}: \; \sigma \in T, \; \mathbf{k} \in \Omega_{\sigma}\}$ is a basis of holomorphic solutions of $\M_A (\beta)$ in the open set $U_T ( R )\cap (\CC^{\ast})^n$ for some $R>0$.
\end{theorem}

For a regular triangulation $T$, set $U_T :=U_T (R)$ for some $R>0$ as in Theorem \ref{key2theorem}.

\section{Local monodromy computation}\label{Sec4}

Let us denote by $\mathscr{S}_A$ the singular locus of $\M_A (\beta )$. A hyperplane $x_j=0$ is contained in $\mathscr{S}_A$ if and only if $a_j$ is a vertex of the polytope $\Delta_A$. Since reordering the variables is equivalent to reordering the columns of $A$ we will assume for simplicity and 
without loss of generality that $a_n$ is a vertex of $\Delta_A$ and we will study the local monodromy of the solutions of $M_A (\beta)$ around $\{x_n=0\}$. Let us consider a complex line $ \mathbb{L}_{c}:=\{x\in \C^n: \; x_j= c_j , \; 1\leq j \leq n-1\}$ with $c=(c_1 ,\ldots , c_{n-1})\in \C^{n-1}$ such that $\mathbb{L}_{c}$ intersects 
$\mathscr{S}_A$ in at most a finite number of points and $(c,0):=(c_1,\ldots , c_{n-1},0)$ does not belong to a component of $\mathscr{S}_A$ different from $\{ x_n = 0\}$. We consider the loop $\gamma_{\varepsilon, c}$ parametrized by $x_j = c_j$ for $1\leq j \leq n-1$ and $x_n =\varepsilon e^{2 \pi i \theta}$, $\theta \in [0,1]$, for $\varepsilon >0$ small enough so 
that $\mathbb{L}_{c}\cap \mathscr{S}_A \setminus \{x_n =0\} \subseteq \{x\in \mathbb{L}_{c} :\; |x_n |>\varepsilon \}$. We include a proof of the following known result for the sake of completeness.

\begin{lemma}\label{char-poly-lemma}
The characteristic polynomial of the local monodromy with respect to $\gamma_{\varepsilon, c}$ of the solutions of a holonomic $D$--module with singular locus $\mathscr{S}$ does not depend on $c$ or $\varepsilon$ chosen as above.
\end{lemma}

\begin{proof}
It is clear that for fixed $c$ as above the monodromy matrix will be independent of $\varepsilon>0$ small enough. Let $d$ be another point chosen as $c$. If $Z$ denotes the union of all the irreducible components of 
$\mathscr{S}$ different from $\{x_n=0\}$, then $(c,0)$ and $(d,0)$ belong to $\{x_n=0\}\setminus Z$, which is connected. Let $\alpha: [0,1]\longrightarrow \{x_n=0\}\setminus Z$ be a path from $(c,0)$ to $(d,0)$. Then the family of loops 
$\{\gamma_{\varepsilon, c'}: \; (c',0)\in \alpha ([0,1]) \}$ is a continuous family with respect to $c'$ and since $\C^n \setminus Z$ is an open set we can choose $\varepsilon >0$ small enough so that none of these loops intersect $\mathscr{S}$.
In particular one can deform continuously $\gamma_{\varepsilon, c}$ into $\gamma_{\varepsilon, d}$ in $\C^n \setminus \mathscr{S}$ in such a way that the point  $\gamma_{\varepsilon, c}(0)$ is transformed into $\gamma_{\varepsilon, d}(0)$ along a path 
$\delta$ such that $\delta (t)= \gamma_{\varepsilon, c'} (0)$ whenever $(c',0)=\alpha (t)$. Assume that $M_c$ is the monodromy matrix corresponding to $\gamma_{\varepsilon, c}$ (for a fixed basis of solutions), 
$M_d$ is the monodromy matrix corresponding to $\gamma_{\varepsilon, d}$ (for another fixed basis of solutions) and $C$ is the connecting matrix (between those bases) for the path  
$\delta$. Then we have that $C^{-1} M_d C= M_c$ and thus $M_d$ and $M_c$ have the same characteristic polynomial.
\end{proof}

Let us set some more notation in order to state the main result in this section. We consider the regular subdivision $T_0$ (see Notation \ref{regular-subdivision-T0}) and the set $$\Sigma =\{ \tau \in T_0 :\; n \in \tau \}.$$ 

Each $\tau \in \Sigma$ is of the form $\Gamma (\tau) \cup \{ n\}$ for some $(d-2)$-dimensional face $\Gamma (\tau)$ of $T_0$ by Remark \ref{T0description}. 
Let $\rho (\tau )$ be the primitive inner conormal vector of the facet $\Delta_{\Gamma (\tau )}$ of the polytope $\Delta_{\tau}$ and set $h (\tau )=\langle \rho (\tau ), a_n\rangle >0$.
    
We obtain the following theorem for the local monodromy around $x_n=0$, which is reminiscent of the main theorem in \cite{AET} for the local monodromy around $x_n =\infty$. 

\begin{theorem}\label{main-theorem}
If the parameter vector $\beta \in \CC^d$ is nonresonant then the characteristic polynomial of the local monodromy of the solutions of $\cM_A (\beta)$ around $x_n =0$ is 
$$\lambda_{0} (z)= (z-1)^{\operatorname{vol}_{\Z^d} (A)- \sum_{\tau \in \Sigma} \operatorname{vol}_{\Z^d} (\tau)} \prod_{\tau \in \Sigma} ( z^{h (\tau)}- e^{2 \pi i \langle  \rho (\tau) , \beta \rangle } )^{\operatorname{vol}_{\Z^d}(\tau)/ h (\tau)}$$
\end{theorem}

\begin{proof} 
By Lemma \ref{char-poly-lemma} we are allowed to choose any convenient $c$ and $\varepsilon>0$ as above in order to compute the characteristic polynomial. 
Let $T$ be any regular triangulation of $A$ that refines $T_0$ and assume first that $\beta\in \CC^d$ is very generic. By Lemma \ref{key-Lemma}, the open set $U_T $ intersects the hyperplane $x_n=0$. We choose $(c_1, \ldots, c_{n-1},c_n )\in U_T \cap (\CC^{\ast})^n$ and $0<\varepsilon <|c_n|$ so that the loop $\gamma_{c,\varepsilon}$ is contained in $U_T \cap (\CC^{\ast})^n$.
On the other hand, from Theorem \ref{key2theorem} there is a fundamental set of $\operatorname{vol}_{\Z^d}(A)$ many (multivalued) holomorphic solutions that can be written as $\Gamma$--series in the open set $U_T\cap (\CC^{\ast})^n$. 
By analytic continuation along $\gamma_{\varepsilon, c}$, each series $\phi_{\sigma}^{\mathbf{k}}$ is transformed into $e^{2 \pi i (v_{\sigma}^{\mathbf{k}})_n } \phi_{\sigma}^{\mathbf{k}}$. Hence, we just need to show that the roots of the polynomial 
$\lambda_0 (z)$ are exactly $\{e^{2 \pi i (v_{\sigma}^{\mathbf{k}})_n }: \; \sigma \in T , \mathbf{k} \in \Omega_{\sigma} \}$ where each root is repeated as many times as its multiplicity. Notice that both the degree of $\lambda_0$ and the cardinality of the pairs 
$(\sigma, \mathbf{k})$ are equal to $\operatorname{vol}_{\Z^d}(A)$.

Indeed, for all $\tau \in T_0 \setminus \Sigma$ and for all $\sigma \in T$ with $\sigma \subseteq \tau$ we have $(v_{\sigma}^{\mathbf{k}})_n = k_n \in \Z$ and thus $e^{2 \pi i (v_{\sigma}^{\mathbf{k}})_n }=1$. This corresponds with the factor
$(z-1)^{\operatorname{vol}_{\Z^d} (A)- \sum_{\tau \in \Sigma} \operatorname{vol}_{\Z^d} (\tau)}$ of $\lambda_0$. Let us consider now $\tau \in \Sigma$ and $\sigma \in T$ with $n\in \sigma \subseteq \tau $. 
Let $\alpha >0$ be the smallest positive integer such that $\alpha e_n A_{\sigma}^{-1} \in \Z^d$ where $e_n=(0,\ldots, 0, 1)\in \R^{\sigma}$. It is clear that $\rho (\tau)=  \alpha e_n A_{\sigma}^{-1}$ and that $h(\tau)=\alpha$. On the other hand, we have 
$(v_{\sigma}^{\mathbf{k}})_n = e_n A_{\sigma}^{-1}(\beta - A_{\overline{\sigma}} \mathbf{k})=\dfrac{1}{h(\tau)}\langle \rho (\tau), \beta-A_{\overline{\sigma}} \mathbf{k} \rangle$.
Since the primitive vector $\rho (\tau )$ induces a surjective morphism of abelian groups from $\Z^d / \Z A_{\sigma}$ to $\Z /h(\tau)\Z$ we have that 
its kernel has cardinality equal to $\operatorname{vol}_{\Z^d}(\sigma) / h (\tau)$. Since 
$\sum_{\tau \supseteq \sigma \in T} \operatorname{vol}_{\Z^d} (\sigma )= \operatorname{vol}_{\Z^d} (\tau)$ we get the result for very generic parameters. 
On the other hand, from the proof of \cite[Corollary 3.3]{AET} (which does not use rapid decay cycles of \cite{ET}) we have that the local monodromy matrix depends holomorphically on nonresonant parameters.
\end{proof}

\begin{remark}
By \cite{DMM}, running the canonical series algorithm introduced in \cite[Chapter 2.5.]{SST} for a weight vector $\w$ such that $(1,\ldots, 1) \in C(T_{\w})$ produces a basis of convergent series solutions of $M_A (\beta)$ for any $\beta\in \C^d$. In particular,
this holds if $\w_0 \in C(T_{\w})$ and in this case these series are convergent in the open set $U_{T_{\w}}$. Thus, by Lemma \ref{key-Lemma}, one 
could use these series in order to compute the characteristic polynomial $\lambda_{0} (z)$ for any $\beta$.
\end{remark}

\begin{remark}
In \cite{RSW} the authors compute an upper bound for the set of roots of the $b$--function of $M_A (\beta)$ for the restriction at $x_j=0$. In the case when $M_A (\beta )$ is regular holonomic, if $\alpha$ is a root of the $b$--function then $e^{2 \pi i \alpha}$ is 
a root of $\lambda_0 (z)$, but in the irregular case the roots of this $b$--function might be related with monodromy of non convergent $\Gamma$--series instead. On the other hand, their description uses \emph{quasidegrees} of certain \emph{toric modules} instead of a polyhedral 
subdivision of $A$, so it does not seem obvious to compare their candidates to roots with the roots of $\lambda_0 (z)$. 
\end{remark}

\begin{remark}\label{remark-infinity}
Consider the vector $\w_{\infty}= (1,\ldots, 1, 1-\varepsilon)$ for $\varepsilon >0$ small enough. 
Take $T_{\infty}$ to be the regular subdivision of $A$ induced by $\w_{\infty}$, then $T_{\infty}$ is a refinement of $\Gamma_A$. If we substitute $T_0$ by $T_{\infty}$, $\w_0$ by $\w_{\infty}$ and consider loops of the form 
$\gamma_{c,\nu}$ with $\nu>0$ big enough so that $\mathbb{L}_{c}\cap \mathscr{S}_A \subseteq \{x\in \mathbb{L}_{c} :\; |x_n |< \nu \}$ and with the opposite orientation, then for any regular triangulation that refines $T_{\infty}$ we have that $U_T$ contains loops of this type (the proof is 
analogous to the proof of Lemma \ref{key-Lemma}). With these modifications, the proof of Theorem \ref{main-theorem} gives an alternative combinatorial proof of the main theorem in \cite{AET}. 
\end{remark}

\begin{example}\label{example}
Let us consider the matrix $$A=\left(\begin{array}{ccccc}
                                      1 & 1 & 1 & 1 & 1\\
                                      0 & 1 & 0 & 1 & 2\\
                                      0 & 0 & 1 & 1 & 2\end{array}\right)$$ and a nonresonant parameter vector $\beta \in \CC^3$. 
                                       
\begin{figure}[h]\label{Figure}
\centering
\includegraphics[scale=0.4]{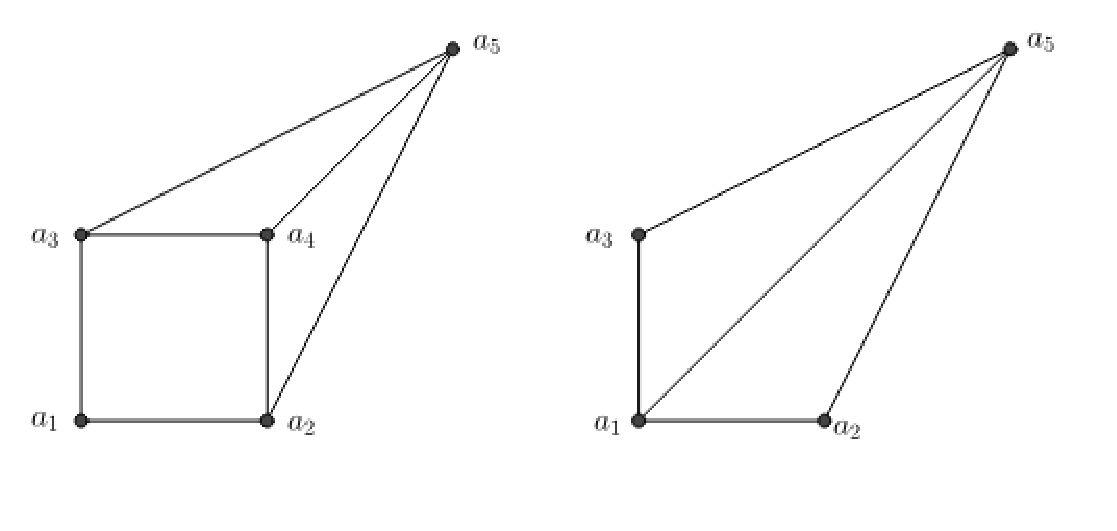}
\caption{Regular subdivisions $T_0$ (left) and $T_{\infty}$ (right) of $A$.}
\end{figure}

The facets of the regular subdivision $T_0$ are $\tau_1=\{1,2,3,4\}$, $\tau_2=\{2,4,5\}$ and $\tau_3=\{3,4,5\}$ (see Figure 1). We have that  $\operatorname{vol}_{\Z A} (\tau_1)=2$, $\Sigma =\{ \tau_2 , \tau_3 \}$, 
$\operatorname{vol}_{\Z A } (\tau)=1$ and $h(\tau)=1$ for $\tau \in \Sigma$, $\rho (\tau_2)=(-1,1,0)$ and $\rho (\tau_3 )=(-1,0,1)$. The characteristic polynomial for the monodromy of the solutions of $M_A (\beta )$ around $x_5=0$ is given by 
$$\lambda_0 (z)= (z-1)^2 (z-e^{2 \pi i (\beta_2- \beta_1)})(z-e^{2 \pi i (\beta_3 -\beta_1)}).$$ On the other hand, the facets of $T_{\infty}$ are $\{1,2,5\}$ and $\{1,3,5\}$. The characteristic polynomial $\lambda_{\infty}$ for the monodromy around $x_5 =\infty$ is $$\lambda_{\infty} (z)=(z^2-e^{-2 \pi i \beta_2})(z^2-e^{-2 \pi i \beta_3}).$$

\end{example}

\section{Monodromy invariant subspaces corresponding to facets of $\Gamma_A$}\label{Sec5}

From \cite[Theorem 4.1]{BMW} an $A$--hypergeometric function can only have singularities along the zero set of the product of the principal $\tau$--determinants $\mathcal{E}_{\tau}=\mathcal{E}_{A_\tau}\in \C [x_{\tau}]:=\C [x_j : \; j\in \tau ]$ for $\tau$ varying between the different facets of $\Gamma_A$ 
(see \cite[Chapter 10, Equation 1.1]{GKZ3} for the definition of $\mathcal{E}_{A}$). The following conjecture would provide a refinement of this result:

\begin{conjecture}\label{monodromy-invariant-subspaces-conjecture}
If $\tau$ is a facet of $\Gamma_A$ and $\beta$ is nonresonant, then the space of holomorphic solutions of $M_A (\beta)$ at a nonsingular point has a monodromy invariant subspace $S_{\tau}$ of dimension $\operatorname{vol}_{\Z^d}(\tau)$ that 
can only have singularities along the zero set of the principal $\tau$--determinant $\mathcal{E}_{\tau} $. If $T$ is any regular triangulation refining $\Gamma_A$ and $\beta\in \C^d $ is very generic, then we can take $S_{\tau}$ 
the space generated by $\{\phi_{\sigma}^{\mathbf{k}}: \; \sigma\in T (\tau), \; \mathbf{k} \in \Omega_{\sigma} \}$ where $T(\tau):=\{ \sigma \in T:\; \sigma \subseteq \tau\}$.
\end{conjecture}

We prove a particular case of this conjecture.

\begin{prop}\label{monodromy-invariant-subspaces}
Conjecture \ref{monodromy-invariant-subspaces} holds if $\tau$ is a facet of $\Gamma_A$ containing only $d$ columns of $A$.
\end{prop}

\begin{proof}
From the proof of \cite[Corollary 3.3]{AET} we can assume that $\beta$ is very generic. By the assumption, any regular triangulation $T$ refining $\Gamma_A$ satisfies that $T(\tau)=\{\tau\}$.

Notice also that each $\Gamma$-series $\phi_{\tau}^{\mathbf{k}}$ is convergent at any point of $U=\C^{\overline{\tau}}\times (\C^{\ast})^{\tau}$. In particular, it defines a multivalued function with singularities only around the hypersurface $\{\prod_{j\in\tau} x_j =0\}$. Since the fundamental group of the complement of $\cup_{j\in \tau} \{ x_j =0\}$ is the free group generated by one loop around each $x_j=0$ for $j\in \tau$, it is enough to consider the monodromy action with respect to these kind of loops. It is clear that the analytic continuation of $\phi_{\tau}^{\mathbf{k}}$ with respect to a loop around $x_j=0$ is given by $e^{2 \pi i (A_{\tau}^{-1}(\beta-A_{\overline{\tau}}\mathbf{k}))_{j}}\phi_{\tau}^{\mathbf{k}}$. Thus, it follows that the space $S_{\tau}$ generated by $\{\phi_{\tau}^{\mathbf{k}}: \;\mathbf{k} \in \Omega_{\tau} \}$ is monodromy invariant and it is a subspace of dimension $\operatorname{vol}_{\Z^d}(\tau)$ of the space of solutions of $M_A (\beta)$.
\end{proof}

\begin{cor}\label{one-dimensional-invariant}
If there are no more than $d$ columns of $A$ in any facet of $\Gamma_A$ and $\beta\in \C^d$ is nonresonant, then the solution space of $M_A (\beta)$ is a direct sum of one--dimensional (global) monodromy invariant subspaces. 
\end{cor}

\begin{proof}

The proof of Proposition \ref{monodromy-invariant-subspaces} can be applied to each facet of $\Gamma_A$ in this case. In particular, for very generic $\beta\in \C^d$ the set of $\Gamma$--series given in Theorem \ref{key2theorem} for $T=\Gamma_A$ is a basis of (multivalued) holomorphic functions in $U_T \cap (\C^{\ast})^n = (\C^{\ast})^n$ and the $\C$-linear space generated by each $\Gamma$-series is monodromy invariant. Thus, from the proof of \cite[Corollary 3.3]{AET} we get the result for nonresonant $\beta\in \C^d$.

%
\end{proof}

The following lemma generalizes \cite[Proposition 6.8]{Saito}.

\begin{lemma}\label{lemma-reducibility}
Let $M$ be a holonomic $D$--module and $M(x):=\C (x)\otimes_{\C [x] } M$. Then $M(x)$ is reducible as a module over $D (x)=\C (x)\otimes_{\C [x] } D$ if and only if $M$ has a quotient $D$--module with holonomic rank between $1$ and $\operatorname{rank}(M)-1$. 
In particular, if $M(x)$ is reducible then the solution space of $M$ has a proper monodromy invariant subspace.
\end{lemma}
\begin{proof}
The if direction holds because $\C(x)\otimes_{\C [x]}$ is a right exact functor and the fact that $\operatorname{rank}(N)=\dim_{\C (x)}(N(x))$ for any holonomic $D$--module $N$, which is a well known result due to Kashiwara. Let us prove the 
converse. Since $M$ is holonomic we may assume that $M=D/I$ for some left ideal $I\subseteq D$. If $M(x)$ is reducible then it has a proper submodule of the form $N=J/D(x)I$ for some left ideal $J\subsetneq D(x)$. Moreover, $1\leq \dim_{\C (x)}(N)< \dim_{\C (x)}(M)=\operatorname{rank}(M)$.
Hence the quotient $M(x)/N$ is isomorphic to $D(x)/J$ and its $\C(x)$--dimension is between $1$ and $\operatorname{rank}(M)-1$. We have that $I \subseteq D\cap D(x)I\subseteq \widetilde{J}:=D\cap J \subsetneq D$ and so $D/\widetilde{J}$ is a quotient of 
$M$. Moreover, it is obvious that $\widetilde{J}(x)=J$. Thus, $\operatorname{rank}(D/\widetilde{J})=\dim_{\C (x)} (D(x)/J)$ is between $1$ and $\operatorname{rank}(M)-1$. The last statement follows from the fact that the solution space of any quotient of $M$ 
is a subspace of the solutions of $M$.
\end{proof}

\begin{remark}
Let us point out that there are reducible holonomic $D$--modules without quotients of nonzero smaller rank. For example, $M=D/Dx\partial$ is reducible and $N=D/D\partial$ is a quotient of $M$, but $M(x)=N(x)$ is an irreducible 
$D(x)$--module. 
\end{remark}

\begin{remark}
For any holonomic reducible $D$--module $M$ such that $M(x)$ is irreducible there is an irreducible subquotient $N$ of $M$ with the same 
rank such that the solution spaces of $M$ and $N$ at nonsingular points are isomorphic (although their solutions complexes are not). If $N$ is regular holonomic then its irreducibility is equivalent to the irreducibility of its solution complex by the Riemann-Hilbert correspondence (cf. \cite{Kas84}, \cite{Meb84}). 
In particular, if $M$ is regular holonomic and $M(x)$ is irreducible then its solution space does not have any proper monodromy invariant subspace. However this is not true in general when $M$ is not regular holonomic. 
In fact, we will see in in Corollary \ref{one-dimensional-invariant} that there is an infinite family of irregular hypergeometric systems $M_A (\beta)$ with a proper monodromy invariant subspace even when $M_A (\beta)(x)$ is irreducible.
\end{remark}

Let $G$ be a minimal set of columns of $A$ such that $\operatorname{pos}(G)$ is a face of $\operatorname{pos}(A)$ and 
$\operatorname{rank}_{\Z} (\Z G)+ \operatorname{card}(\overline{G})=d$. Then, by \cite[Lemma 3.7. (10)]{SW2} it is equivalent to study the solutions of $M_A (\beta)$ 
and the solutions of $M_G (\beta_G )$ where $\beta=\beta_G + \beta_{\overline{G}}$ and $\beta_G \in \C G$ and $\beta_{\overline{G}}\in \C \overline{G}$ are unique.

\begin{theorem}\label{reducible-monodromy-representation}
If Conjecture \ref{monodromy-invariant-subspaces-conjecture} holds, the solution space of $M_A (\beta)$ at any nonsingular point has reducible monodromy representation if and only if at least one of the following conditions holds:
\begin{enumerate}
 \item[i)] $\beta_G \in \C^d$ is resonant for $G$.
 \item[ii)] There is no hyperplane off the origin containing all the columns of $A$.
\end{enumerate}
\end{theorem}
\begin{proof}
If $G=\emptyset$ we consider by convention that $\beta_G=0$ is nonresonant for $G$. In this case $\operatorname{rank}(M_A (\beta))=\operatorname{vol}_{\Z^d}(A)=1$ and so the space of solutions has irreducible monodromy representation. 
By \cite[Lemma 3.7. (10)]{SW2} we can assume without loss of generality that $G=A$ in order to simplify the proof (otherwise everything would be written for $\beta_G$ and $G$ instead of 
$\beta$ and $A$ respectively). 

Let us prove first the if direction. If $\beta$ is resonant, then by the proof of \cite[Theorem 4.1]{SW2} we have that $M_A (\beta)(x)$ is a reducible  
$D(x)$--module and it is enough to use Lemma \ref{lemma-reducibility} in this case. Thus we can assume now that $\beta\in \C^d $ is nonresonant and that ii) holds.
If $\Gamma_A$ has at least two facets and $\tau$ is one of them, since we assume that Conjecture \ref{monodromy-invariant-subspaces-conjecture} holds, there exists a proper 
monodromy invariant subspace of solutions of $M_A (\beta)$ of dimension $\operatorname{vol}_{\Z^d}(\tau)<\operatorname{vol}_{\Z^d}(A)= \operatorname{rank}(M_A (\beta))$.

We can assume now that $\Gamma_A=\tau$ is a facet and that all the columns of $A$ that do not belong to $\tau$ belong to $\Delta_A \setminus \tau$. 
Then the variety $Y_{\tau}=\{x_j =0 :\; \forall j\notin \tau\}$ is non--characteristic for $M_{A}(\beta)$. Thus, by \cite[Theorem 2.3.1]{Kas70}, the restriction of the solution sheaf of $M_A (\beta)$ to $Y_{\tau}$ is isomorphic to the solution sheaf of the $D$--module restriction of $M_A (\beta)$ to $Y_{\tau}$. 
Since $\Z A \neq \Z \tau$, by \cite[Theorem 2.1]{FW}, the restriction of $M_A (\beta)$ to $Y_{\tau}$ is isomorphic to 
$\oplus_{\beta ' \in \Lambda} M_{\tau}(\beta ')$ for certain set $\Lambda$ of cardinality $[\Z A : \Z \tau ]>1$. Moreover, if we denote by 
$ V(\mathcal{E}_{\tau})\subseteq \C^{\tau}$ the zero set in $\C^{\tau}$ of the principal $\tau$--determinant, 
it is clear that $\pi_1 (\C^n \setminus (\C^{\overline{\tau}}\times V(\mathcal{E}_{\tau}))\simeq \pi_1 (\C^{\tau} \setminus  V(\mathcal{E}_{\tau}))$. 
Thus, since the solution space of $\oplus_{\beta ' \in \Lambda} M_{\tau}(\beta ')$ has reducible monodromy representation we have that the solution space of 
$M_A (\beta)$ does too. 

For the converse, assume that i) and ii) are false, then $A=\tau$, $M_A (\beta)$ is regular holonomic by \cite{Hotta} and the solution space of $M_A (\beta)$ has irreducible monodromy representation for nonresonant parameters by \cite[2.11 Theorem]{GKZ2}. 
\end{proof}

\vspace{.5cm}

\noindent \textit{E-mail address:} \texttt{mcferfer@us.es}

\noindent Departamento de \'Algebra \& Instituto de Matem\'aticas de la Universidad de Sevilla (IMUS), Universidad de Sevilla.

\end{document}